\documentclass{article}

\usepackage[english]{babel}
\usepackage[letterpaper,top=2cm,bottom=2cm,left=3cm,right=3cm,marginparwidth=1.75cm]{geometry}

\usepackage{amsmath}
\usepackage{amsthm}
\usepackage{amsfonts}
\usepackage{graphicx}
\usepackage{amssymb}
\usepackage{bbm}
\usepackage[colorlinks=true, allcolors=blue]{hyperref}
\usepackage[style=apa, backend=biber]{biblatex}
\usepackage{comment}

\newtheorem{theorem}{Theorem}
\numberwithin{theorem}{section}
\newtheorem{lemma}[theorem]{Lemma}
\newtheorem{proposition}[theorem]{Proposition}
\newtheorem{remark}[theorem]{Remark}
\newtheorem{corollary}[theorem]{Corollary}

\newcommand{\N}{{\mathbb{N}}}
\newcommand{\Z}{{\mathbb{Z}}}
\newcommand{\PP}{{\mathbb{P}}}

\newcommand{\G}{{\mathcal{G}}}
\newcommand{\M}{{\mathcal{M}}}

\newcommand{\x}{\xi^{\ast}}
\numberwithin{equation}{section}

\title{\textbf{Phase transitions for percolation of words in one dimension}}
\author{Gustavo O. de Carvalho\thanks{
Departamento de Estatística,
Universidade Estadual de Campinas, Brazil. \texttt{godc@ime.unicamp.br}
} \and Pablo A. Gomes\thanks{
Departamento de Estatística,
Universidade de São Paulo, Brazil. \texttt{pagomes@usp.br}
}}
\date{}

\begin{document}
\maketitle
\textit{}
\begin{abstract}
   In this paper, we investigate models of percolation of words on $\mathbb{Z}_+$ with long-range connections. The underlying graph is oriented and constructed according to a sequence of non-negative ranges. In the first model, the sequence of ranges is considered random and, in the second model, deterministic and non-decreasing. In both cases, we establish a phase transition for the occurrence of percolation of all words simultaneously from the origin.
\end{abstract}


\medskip


\section{Introduction}

\subsection{Preliminaries and Motivation}

To introduce a general framework of percolation of words, we start with some definitions. The alphabet consists of two letters: $0$ and $1$. A word is a semi-infinite sequence of letters of the alphabet $\xi = (\xi_1, \xi_2, \dots) \in \{0,1\}^{\mathbb{N}}$. The set of all words is denoted by \[\Xi = \{0,1\}^{\mathbb{N}}.\]

The model of percolation of words is defined on an underlying graph. In general, given a graph $G=(V,E)$, a letter is independently assigned to each vertex $v \in V$. Given a parameter $p \in (0,1)$, let $(X_v)_{v\in V}$ be a family of independent Bernoulli random variables with probability $p$ of success. For each $v \in V$, $X_v$ denotes the letter assigned to $v$.

We say that a word $\xi \in \Xi$ is seen from a vertex $v_0 \in V$ if there exists a sequence of distinct vertices $v_1,v_2,\dots$ of $V$ such that $X_{v_i} = \xi_i$ and $(v_{i-1}, v_{i}) \in E$ for every $i = 1, 2, \dots$. The set of words seen from a set of vertices $S \subset V$ is defined by
\[
W_S = \{ \xi \in \Xi ~:~ \xi \textrm{ is seen from some }v \in S \}.
\]
For simplicity, when $S = \{v\}$, we denote $W_S$ by $W_v$.

Our main object of interest is the probability that all words are seen simultaneously from the origin. According to our notation, this event is denoted by 
\(\{W_0 = \Xi\}\).

In this paper, we focus on one-dimensional underlying graphs. See also Section~\ref{relatedworks} for references to higher-dimensional models. We proceed by defining the 
long-range graph $\mathcal{G} = (\mathbb{Z_+, \mathcal{E}})$ considered in our models. The set of oriented edges  $\mathcal{E}$ is constructed according to the sequence of ranges 
\[
\mathcal{M}=(M_n)_{n\in \mathbb{Z}_+}.
\] 
Given a vertex $n \in \mathbb{Z}_+$, the set of oriented edges starting at $n$ is given by 
\begin{equation}\label{eq:En}
    \mathcal{E}_n = \{(n,n+i) ~:~ 0 <  i \leq M_n\}.
\end{equation}
To conclude the definition of $\mathcal{G}$, we set $\mathcal{E} = \cup_{n \in \mathbb{Z}_+} \mathcal{E}_n$.

Having provided a brief description above, we now discuss the motivations for this work. In \cite{grimmet_liggett_richthammer_2010}, the problem of percolation of words is studied on the same graph with constant range, that is, $M_n=M<\infty$ for every $n\in\mathbb{Z}_+$. The authors mainly investigate the probability of observing finite words in $\G$. One motivation for studying finite words is that any periodic infinite word $\xi\in\Xi$ is seen in $\G$ with probability zero, since $(X_n)_{n\in\mathbb{Z}_+}$ almost surely contains arbitrarily long consecutive runs of both $0$'s and $1$'s.

Although it is not possible to observe all words when $M_n=M<\infty$ for every $n\in\mathbb{Z}_+$, \cite{basu_sly_2014} showed that \textit{almost all words} are seen somewhere. More precisely, let $\eta\in\Xi$ be a random word generated by the measure $\mu_q$ under which \((\xi_n)_{n\in\mathbb{Z}_+}\) is an independent sequence of Bernoulli$(q)$ random variables independent of $(X_n)_{n\in\mathbb{Z}_+}$. They proved that there exists $M$ large enough such that the event $\{\eta\in\Xi:P(\eta\in W_{\mathbb{Z}_+})=1\}$ occurs $\mu_q$-a.s. for $p=q=1/2$. By \cite[Theorem 13]{grimmet_liggett_richthammer_2010}, the same conclusion can be extended for any $(p,q)\in(0,1)^2$.

In this context, it is natural to ask how this behavior changes when considering long-range graphs with edges of unbounded length. Our main goal is to find conditions on the long-range graph $\mathcal{G}$ under which all words (not just almost all words) are seen. Note by the discussion above that this requires the sequence $(M_n)_{n\in\mathbb{Z}_+}$ to be unbounded. Accordingly, we consider two different settings: random i.i.d. ranges and deterministic non-decreasing ranges. Precise definitions and results are  given in the next section.

\subsection{Models and Results}

\bigskip

\noindent {\textbf{ \large I.i.d. ranges}}
\medskip

Fix a non-negative integer-valued random variable $R$. We assume that $\mathcal{M}=(M_n)_{n\in \mathbb{Z}_+}$ is a sequence of independent random variables with the same distribution as $R$. 

In this case, we are considering the model of percolation of words on a random graph, thus there are two layers of randomness: one with respect to the ranges $\mathcal{M}=(M_n)_{n\in \mathbb{Z}_+}$ and the other with respect to the letters $(X_n)_{n \in \mathbb{Z}_+}$ assigned to the vertices. The law with respect to the ranges is denoted by $P$. Given a realization of the random graph $\G$, we denote by $P_{\G}$ the quenched law. The natural annealed law is considered and denoted by $\PP$. Given two measurable sets $A$ and $B$ in the standard $\sigma-$algebra generated by cylinders with respect to $(X_n)_{n \in \mathbb{Z}_+}$ and $(M_n)_{n\in \mathbb{Z}_+}$ respectively, we have
\[
\PP (A \times B) = \int_{B}P_{\G}(A)dP(\G).
\]
With some abuse of notation, $\PP(A)$ denotes $\PP(A\times B )$, where $B$ is any measurable set such that $P(B) = 1$.

Regarding the tail distribution of $R$, we establish the following phase transition.
\begin{theorem}\label{theo:random_m} Given the parameter $p \in (0,1)$ and the distribution of $R$, the model of percolation of words on the random graph $\G$ is such that
\begin{enumerate}
\item If \( \displaystyle \limsup_{n\to\infty}nP(R\geq n)< \frac{1}{\min\{p,1-p\}},\) then \(\PP(W_0 = \Xi) = 0,\)

 \item   If
\( \displaystyle \liminf_{n\to\infty}nP(R\geq n)>\frac{3}{\min\{p,1-p\}},
\)
then
\(\PP(W_0 = \Xi) > 0.\)
\end{enumerate}
\end{theorem}

\begin{remark}\label{remark:individual_words}
It also holds that \(\PP(W_0 = \Xi) = 0\) whenever \(nP(R\geq n)= \frac{1}{\min\{p,1-p\}}\) for all sufficiently large \(n\). Moreover, when \(\liminf_{n\to\infty}nP(R\geq n)> \frac{1}{\min\{p,1-p\}}\), then \(\PP(\xi\in W_0) > 0\) for any $\xi\in\Xi$. Therefore, \(\PP(\xi\in W_{\mathbb{Z}_+}) =1\) by translational invariance, so every word $\xi\in\Xi$ is seen somewhere, but not necessarily all words are seen from the same vertex. In particular, this also implies that almost all words are seen somewhere, i.e., both $\{\eta\in\Xi:\PP(\eta\in W_0)>0\}$ and $\{\eta\in\Xi:\PP(\eta\in W_{\mathbb{Z}_+})=1\}$ hold $\mu_q$-almost surely for any $q\in[0,1]$ (see Section~3 of \cite{lima_2008}).
\end{remark}

As a consequence of Theorem~\ref{theo:random_m} together with the translational invariant properties of the model, the following quenched result is also derived.

\begin{corollary}    
\label{theo:random_mQuenched}
     Given the parameter $p \in (0,1)$ and the distribution of $R$, the model of percolation of words in the random graph $\G$ is such that
\begin{enumerate}
\item If \( \displaystyle \limsup_{n\to\infty}nP(R\geq n) < \frac{1}{\min\{p,1-p\}},\) then, for $P-$almost all $\G$, \(P_{\G}(\exists v \in \Z_+ \colon W_v = \Xi) = 0,\)  
\item   If
\( \displaystyle \liminf_{n\to\infty}nP(R\geq n)>\frac{3}{\min\{p,1-p\}},
\)
then, for $P-$almost all $\G$, \(P_{\G}(\exists v \in \Z_+ \colon W_v = \Xi) = 1.\) 
\end{enumerate}
\end{corollary}

\bigskip
\color{black}
\noindent\textbf{\large Non-decreasing ranges}

\medskip

Now we consider the case in which the sequence of ranges $\mathcal{M} = (M_n)_{n\in \mathbb{Z}_+}$ is deterministic, non-decreasing, and such that $M_0 \geq 2$. Observe that if $M_0 < 2$, then the event $\{W_0 = \Xi\}$ has probability zero. Moreover, the monotonicity of \(\M\) has an interesting consequence: whenever a vertex \(x\) sees a vertex \(y\), every vertex \(x'\in\{x+1,x+2,...,y-1\}\) also sees \(y\). To distinguish this setting from the random case, the underlying graph $\G$ in the deterministic case will be denoted by $\bar\G$.

According to the parameter $p \in (0,1)$, consider the word $\xi^{\ast}$ defined by
\begin{equation}\label{eq:least_probable}
    \xi^*:=\begin{cases}
    (0,0,0,...), \text{ if }p\geq 1/2;\\
    (1,1,1,...), \text{ if }p< 1/2.
\end{cases}\end{equation} 
For a general graph $G$, the following proposition states that the word $\xi^{\ast}$ is the worst scenario case to percolate with respect to the occurrence of percolation of a fixed word. 
\begin{proposition}[\cite{lima_2008}{, Lemma 2}]\label{prop:worst}
    Given $p \in (0,1)$, for any graph $G$ and vertex $v\in V$, the model of percolation of words in $G$ is such that
    \[
    P_G(\xi \in W_v) \geq P_G(\xi^{\ast} \in W_v), ~\forall \xi \in \Xi.
    \]
\end{proposition}

Despite Proposition~\ref{prop:worst}, when $\M$ is non-decreasing, an interesting characterization for the occurrence of simultaneous percolation of all words is given in next result.
\begin{theorem}\label{teo:least_probable}
   Given the parameter $p \in (0,1)$, let $\mathcal{M} = (M_n)_{n\in \mathbb{Z}_+}$ be a non-decreasing sequence such that $M_0\geq 2$ and let $\bar\G$ be the graph defined according to \eqref{eq:En}. Then,
\[P_{\bar\G}(\xi^*\in W_0)>0 \iff P_{\bar\G}(W_0=\Xi)>0.\]
\end{theorem}

Through Theorem \ref{teo:least_probable}, it is possible to draw conclusions about the event \(\{W_0=\Xi\}\) by analyzing the percolation of a constant word, which is, in turn, an easier problem to study.

Given a parameter $p \in (0,1)$, consider the auxiliary constant $C$ defined by
\[
C = \frac{1}{\log\Big(\frac{1}{\max\{p,1-p\}}\Big)}.
\]
The next result relies on Theorem \ref{teo:least_probable} and establishes a phase transition with critical value $C$.

\begin{theorem}\label{theo:mdeterministic}
    Given the parameter $p \in (0,1)$ and $\bar\G$ the graph defined according to \eqref{eq:En}. The following items hold:
\begin{enumerate}
    \item If $\displaystyle
\limsup_{n\to\infty} \frac{M_n}{\log(n)} < C$, then
\(P_{\bar\G}(W_0=\Xi)=0,\)

\item If $\displaystyle
\liminf_{n\to\infty} \frac{M_n}{\log(n)} > C$ with $M_n\ge 2$ for all $n\in\mathbb{Z}_+$, then
\(P_{\bar\G}(W_0=\Xi)>0.\)
\end{enumerate}
\end{theorem}

\begin{remark}\label{remark:m_nondecreasing}
It is also possible to show that \(P_{\bar\G}(W_0 = \Xi) = 0\) whenever \(M_n=C \log(n)\) for all sufficiently large \(n\). However, the proof fails for a general $M_n$ with $\displaystyle
\limsup_{n\to\infty} \frac{M_n}{\log(n)} = C$.
\end{remark}

Note that $\mathcal{M}$ does not need to be non-decreasing in Theorem \ref{theo:mdeterministic}. Although this assumption appears in Theorem \ref{teo:least_probable}, we show that it may be ignored since increasing some values in $\mathcal{M}$ cannot decrease \(P_{\bar\G}(W_0 = \Xi)\).

\color{black}

\subsection{Related Works}\label{relatedworks}

The problem of percolation of words was introduced in \cite{benjamini_kesten_1995} who studied the problem on the hypercubic lattice $\mathbb{L}^d$ with $p=1/2$. They showed that \(P_{\mathbb{L}^d}(W_{\mathbb{Z}^d}=\Xi)=1\) for $d\ge 10$ and \(P_{\mathbb{L}^d}(\text{$W_v=\Xi$ for some $v\in\mathbb{Z}^d$})=1\) for $d\ge 40$. More than 20 years later, \cite{nolin_tassion_teixeira_2023} made a great improvement by showing that $P_{\mathbb{L}^d}(W_{\mathbb{Z}^d}=\Xi)=1$ for $d\ge 3$ and $p\in(p_c^{\text{site}}(\mathbb{L}^d),1-p_c^{\text{site}}(\mathbb{L}^d))$, where $p_c^{\text{site}}(\mathbb{L}^d)$ is the critical parameter for site percolation on the $d$-dimensional hypercubic lattice.

There are also some results on percolation of words on other short-range graphs. On the triangular lattice with \(p=1/2\), even though the constant words \((0,0,...)\) and \((1,1,...)\) are not seen with positive probability, \cite{kesten_sidoravicius_zhang_1998} showed that almost all words are seen with respect to the measure \(\mu_q\), \(q\in(0,1)\). Moreover, \cite{kesten_sidoravicius_zhang_2001} showed that ${\text{$W_v=\Xi$ for some $v\in \mathbb{Z}^d$}}$ occurs almost surely when $p\in(p_c^{\text{site}}(\mathbb{L}^2),1-p_c^{\text{site}}(\mathbb{L}^2))$ on the closed-packed graph of the square lattice, i.e., the graph obtained by adding one diagonal to each face of $\mathbb{L}^2$.

There are also several works studying variations of $\mathbb{L}^d (d \geq 2)$ in which long-range edges are allowed, with the main purpose being to investigate the truncation question, i.e., whether there exists some $K\in\mathbb{N}$ such that all words can still be seen when edges of length greater than $K$ are suppressed. In \cite{lima_2008,lima_sanchis_silva_2011} the truncated graph is deterministic, while in \cite{gomes_lima_silva_2022,gomes_lima_silva_2026} it  corresponds to a bond percolation configuration. It is worth noting that, although these models also involve long-range graphs, the cited works assume that each edge is present independently of the others, which differs from the definitions used in this work.

\subsection{Strategy of proof and Organization}

To show that $\mathbb{P}(W_0=\Xi)=0$ in the i.i.d. case, we fix the word $\xi^*$ and show that it is not seen, taking advantage of the fact that $\xi^*$ is both the hardest word to see and the easiest one to analyze. On the other hand, to show that $\mathbb{P}(W_0=\Xi)>0$, we fix a word $\xi\in\Xi$ and analyze the sequence $(C_n(\xi))_{n\ge 1}$. Given $C_n(\xi)$ vertices, we consider those whose states match $\xi_n$ and then use their ranges to determine the set of new vertices seen, whose cardinality is $C_{n+1}(\xi)$. By means of a reparametrization $(b_n)_{n\ge1}$ of time and an increasing sequence $(r_n)_{n\ge1}$, we show that, for any word $\xi$, the probability of $C_{b_{n+1}}(\xi)<r_{n+1}$ conditional on $C_{b_n}(\xi)\ge r_n$ decays exponentially. We use this exponential decay to compensate for the number of words, which also grows exponentially with the number of digits.

For the non-decreasing case, when both $P_{\bar{\mathcal{G}}}((0,0,\ldots)\in W_0)>0$ and $P_{\bar{\mathcal{G}}}((1,1,\ldots)\in W_0)>0$, we show that
$P_{\bar{\mathcal{G}}}({(0,0,\ldots),(1,1,\ldots)}\subset W_0)>0$
by describing this event as the intersection of two independent events, one depending only on a finite initial segment of the sequence $(X_n)_{n\in\mathbb{Z}_+}$ and the other having a positive probability, obtained with the help of Kolmogorov's 0-1 law for tail events. We also show that if both constant words $(0,0,\ldots)$ and $(1,1,\ldots)$ are seen, then we can suitably interlace the vertices to see any word $\xi\in\Xi$. For these results, it is important that $\mathcal{M}$ be non-decreasing, ensuring that whenever a vertex $x$ sees a vertex $y$, every vertex $x'\in{x+1,x+2,\ldots,y-1}$ also sees $y$.

The remainder of the paper is organized as follows. Section \ref{sec:i.i.d.} establishes the results for the i.i.d. case, with the exception of the main technical ingredient of the proof: Proposition~\ref{lemma:exp_word}, an auxiliary result concerning the growth of $C_n(\xi)$. This proposition constitutes the most technically demanding part of the paper, and its lengthy proof is therefore deferred to Section \ref{sec:lemmas}. Section \ref{sec:nondecreasing} contains the proof for the non-decreasing case. It is worth mentioning that the proofs for the i.i.d. and non-decreasing cases follow different strategies and can therefore be read independently.

\section{Proofs for the i.i.d.~case}\label{sec:i.i.d.}

This section is devoted to proving Theorem~\ref{theo:random_m}. For the conclusion of Item~\textit{1}, it is enough to show that $\PP(\x \in W_0) = 0$, where $\x$ is as defined in \eqref{eq:least_probable}. Without loss of generality, during this section, we assume $p < 1/2$.

For each $n \in \Z_{+}$, let $\bar M_n = \mathbbm{1}_{\{X_n = 1\}}M_n$. Since we are investigating the occurrence of the word $\x$ consisting entirely of 1's, in this auxiliary sequence $(\bar M_n)_{n \in \Z_+}$, we are deleting all the oriented edges starting from vertices where the letter $0$ is assigned. 

A notion of cut points is introduced as follows. We say that $0$ is a cut point. We also say that the vertex $n \in \N$ is a cut point if, for each $ j \in \{0, \dots, n-1 \}$, we have $j +\bar M_j < n $. Let $A_n$ denote the event where $n$ is a cut point.

A key observation is that, due to the regenerative nature of cut points, the number of cut points is geometrically distributed with parameter $\mathbb{P}(\x\in W_0)$. Therefore, we obtain the following conclusion:
\begin{equation}\label{eq: number of cut points}
    \PP(\x \in W_0) > 0 \quad \Longleftrightarrow  \quad \sum_{n \geq 0} \PP(A_n) < \infty.
\end{equation}

\begin{proof}[\textbf{Proof of Item 1 of Theorem~\ref{theo:random_m}}]
Recall that we are considering $p < 1/2$. Assume that the common distribution of the ranges $R$ is such that $\limsup_{n\to \infty} nP(R \geq n) < 1/p.$ Observe that, under this assumption, the common distribution $\bar R$ of the i.i.d. sequence $(\bar M_n)_{n \in \Z_+}$ is such that 
\begin{equation}\label{eq gamma}
\limsup_{n\to \infty} nP(\bar R \geq n)  =  \limsup_{n\to \infty} n\big[pP(R \geq n)\big] =\colon \gamma  < 1.
\end{equation}

Estimates for the probability $\PP(A_n)$ of the occurrence of cut points are given.
\begin{equation}\label{eq:p_an}
    \PP(A_n) = \prod_{j = 0}^{n-1} \PP(\bar M_j < n-j) = \prod_{i=1}^{n} \big[1 -P(\bar R \geq i)\big], \quad n \geq 1.
\end{equation}
From the inequality in \eqref{eq gamma}, given $\varepsilon > 0$, we obtain that, for some constant $c > 0$ and every $n \geq 1$, $\PP(A_n) \geq c/n^{\gamma + \varepsilon}.$ Hence, choosing $\varepsilon$ sufficiently small such that $\gamma + \varepsilon \leq 1$, $\sum_{n \geq 0} \PP(A_n) =\infty$ and, consequently, \eqref{eq: number of cut points} yields $\PP(\x \in W_0) = 0$.  In conclusion, we write
\[
\PP(W_0 = \Xi) \leq \PP(\x \in W_0) = 0
\]
to finish the proof. Moreover, note that the results claimed in Remark \ref{remark:individual_words} follow with small modifications. Indeed, when $\liminf_{n\to\infty}n P(\bar R\ge n)>1$, then analogous arguments show that $\sum_{n \geq 0} \PP(A_n) < \infty$, and the extension to every word is due to Proposition~\ref{prop:worst}. If \(P(\bar R\ge n)=1/n\) for sufficiently large $n$, then \(\PP(A_n)=O(1/n)\) since several terms cancel within the multiplication in (\ref{eq:p_an}).
\end{proof}

For the proof of Item~\textit{2}, we start with some definitions. Fix any word $\xi=(\xi_1,\xi_2,...)\in\Xi$. Define the sequence $(i_n(\xi))_{n\ge 0}$ by setting $i_0(\xi):=0$, $i_1(\xi):=M_0$ and recursively for $n\geq 1$,
\[i_{n+1}(\xi):=\begin{cases}
    \max\{i+M_i\mathbbm{1}_{(X_i=\xi_n)}:i\in (i_{n-1}(\xi), i_n(\xi)]\},&\text{if  $i_n(\xi)\neq i_{n-1}(\xi)$};\\
    i_n(\xi),&\text{otherwise}.
\end{cases}\]

We also define
\begin{equation}\label{eq: C}
    C_n(\xi) := i_{n}(\xi) - i_{n-1}(\xi), \text{ for } n \geq 1.
\end{equation}

Observe that, in this construction of the sequence $(C_n(\xi))_n$, the ranges are updated depending on the current digit of the word which is being considered. Note that the word $\xi$ is seen if $C_n(\xi)>0$ for every $n\in \mathbb{N}$. The opposite does not hold for a general word.

Consider also the auxiliary sequence $(i^*_n)_{n\ge 0}$ by setting $i_0^*:=0$, $i_1^*:=M_0$ and recursively for $n\geq 1$:
\begin{equation}
    i_{n+1}^*:=
    \begin{cases}
    \min\left\{\begin{array}{l}\max\{i+M_i\mathbbm{1}_{(X_i=0)}:i\in (i_{n-1}^*, i_n^*]\},\\
    \max\{i+M_i\mathbbm{1}_{(X_i=1)}:i\in (i_{n-1}^*, i_n^*]\}\end{array}\right\},&\text{if } i_n^*\neq i_{n-1}^*;\\
    i_n^*,&\text{otherwise}.
\end{cases}
\end{equation}

Let
\begin{equation}\label{eq: Cast}
    C^*_n:=i_n^*-i_{n-1}^*, \text{ for } n \geq 1.
\end{equation} 
Note that the definition of $C_n^*$ is similar to that of $C_n(\xi)$, but instead of considering the following digit of the word $\xi$ to set $i_{n+1}(\xi)$, we consider the minimum over the two possible digits to set $i_{n+1}^*$. 
Therefore, $W=\Xi$ if $C_n^*>0$ for every $n\in \mathbb{N}$, because both digits are seen in every step. Then, to conclude the proof, it remains to show that $\PP \big(\cap_{n \geq 1} \{C_n^*>0\}\big) = 0.$ 

The parameters $(L, \Phi)$ will be introduced. $L$ is a natural number, and $\Phi = (r_n)_{n \geq 1}$ is an increasing sequence of positive integers. Define an auxiliary sequence $(b_n)_{n\ge 0}$ depending on the parameter $L$ by $b_0:=1$ and
\begin{equation}\label{eq:def_bn}
b_n=b_n(L):=1+\sum_{j=0}^{n-1} (L+j), \text{ for $n\in\mathbb{N}$.}\end{equation}
Also, for each word $\xi \in \Xi$, define the sequence of events $\big(\mathcal{A}_n(\xi)\big)_{n \geq 1}$ by 
\begin{equation}\label{eq: eventobom}
    \mathcal{A}_{n}(\xi):=\{C_{b_{n}}(\xi)\geq r_{n}\}, \text{ for }n\in\mathbb{N}.
\end{equation}

Noticing that $\big( C_n(\xi) \big)_{n\geq 1}$ is a Markov chain with $0$ as an absorbing state, the occurrence of $\mathcal{A}_n$ for all $n \geq 1$ yields that $C_n > 0$ never reaches $0$, which means that $\xi \in W_0$ as argued before.

For the moment, we assume the following proposition, whose proof is postponed to Section~\ref{sec:lemmas}. 

\begin{proposition}~\label{lemma:exp_word}
Given a parameter $p \in (0,1/2]$, suppose that the common distribution of the ranges $R$ satisfies
\[
\liminf_{n\to\infty} nP(R\geq n) > {3}/{p}.
\]
Then there exist a positive constant $\alpha < 1/2$ and parameters $(L,\Phi)$ such that
\begin{equation}\label{eq: cota}
\PP(\mathcal{A}_{n+1}(\xi) \mid C_{b_n}(\xi)=r ) \geq 1 - 3\alpha^{L+n},
\quad \text{for every $\xi\in\Xi$, $r\ge r_n$, and $n \geq 1$.}
\end{equation}
\end{proposition}

Finally, using all the elements developed throughout this section, we can now prove Item~\textit{2}.

\begin{proof}[\textbf{Proof of Item 2 of Theorem~\ref{theo:random_m}}]
  For any $k\in\mathbb{N}$, let
    \( \Xi_k:=\{0,1\}^k \)
denote the set of finite words with length $k$. Recall the sequence $\big(C_j(\xi)\big)_{j \geq 1}$  as defined in \eqref{eq: C}. Observe that $C_j(\xi)$ only depends on the first $j$ digits of $\xi$. With some abuse of notation, given $\eta \in \Xi_j$, $C_j(\eta)$ denotes the value of $C_j(\xi)$ where $\xi \in \Xi$ is any word such that $\eta_i = \xi_i, \forall i \in \{1, \dots, j\}$. Analogously, $\mathcal{A}_j(\eta)$  denotes the corresponding events $\mathcal{A}_j(\xi)$. 

Recall the sequence $(C^{\ast}_j)_{j \geq 1}$ given in \eqref{eq: Cast}. Note that, for any $n\geq 1$, there exists a non-empty set composed by finite words $\big\{\eta \in\Xi_{b_n} ~\colon~ C_{j}(\eta)=C_{j}^*, ~ \forall j\in\{1,...,b_n\}\big\}$. Arbitrarily select $\eta^n$ as one of these words. The word $\eta^n=(\eta^n_1,\eta^n_2, \dots, \eta^n_{b_n})$ is formed by digits that are always responsible for minimizing the following $C^*$, i.e., $\eta^n_j=1$ if $\max\{i+M_i\mathbbm{1}_{(X_i=1)}\text{ and }i\in(i^*_{j-1},i^*_j]\}<\max\{i+M_i\mathbbm{1}_{(X_i=0)}\text{ and }i\in(i^*_{j-1},i^*_j]\}$; $\eta^n_j=0$ when the opposite inequality occurs; $\eta^n_j$ is arbitrarily chosen if the equality holds. We assume that, $\eta^{n+1}$ is selected in a way that it agrees with $\eta^{n}$ in its first $b_n$ digits.

Since $\cap_{n \geq 1} \mathcal{A}_n(\eta^n) \subset \cap_{n \geq 1} \{C^\ast_n > 0\}$, it remains to show that $\PP\big(\cap_{n \geq 1} \mathcal{A}_n(\eta^n) \big) > 0$. 

For each $n \geq 1$, consider the set 
\[
\Xi^*_{b_{n+1}}:=\{(\xi_1,...,\xi_{b_{n+1}})\in \Xi_{b_{n+1}}:(\xi_1,...,\xi_{b_n})=\eta^n\}
\]
constituted of words of length $b_{n+1}$ that agree with $\eta^n$ in its first $b_n$ digits. In particular, recalling the definition of the sequence $(b_n)_{n \in \N_0}$ in \eqref{eq:def_bn}, we have 
\begin{equation*}
    |\Xi^*_{b_{n+1}}|=2^{b_{n+1}-b_n}=2^{L+n}.
\end{equation*}
Therefore, from Proposition~\ref{lemma:exp_word},Markov property, and union bound, the following inequality is derived
\begin{equation*}
    \PP \bigg( \mathcal{A}_{n+1}(\eta^{n+1}) ~\bigg\vert~ \bigcap_{k=1}^n\mathcal{A}_k(\eta^k) \bigg) 
    \geq 
    1 - \sum_{\xi \in \Xi_{b_n+1}^\ast} \PP \bigg( \mathcal{A}^c_{n+1}(\xi) \bigg\vert \bigcap_{k=1}^n\mathcal{A}_k(\eta^k) \bigg) \geq 1 - 3(2\alpha)^{L+n}. 
\end{equation*}
This upper bound yields the desired estimate for the probability $\PP\big(\cap_{n \geq 1} \mathcal{A}_n(\eta^n) \big)$. In fact, we have
\begin{equation*}
    \PP\big(\cap_{n \geq 1} \mathcal{A}_n(\eta^n) \big) = 
    \prod_{n = 1}^{\infty} \PP \bigg( \mathcal{A}_{n+1}(\eta^{n+1}) ~\bigg\vert~ \bigcap_{k=1}^n\mathcal{A}_k(\eta^k) \bigg) \PP\big(\mathcal{A}_1(\eta^1)\big) \geq 
    \PP\big(\mathcal{A}_1(\eta^1)\big)  \prod_{n = 1}^{\infty} [1 - 3(2\alpha)^{L+n} ].
\end{equation*}
Since $\PP\big(\mathcal{A}_1(\eta^1)\big) > 0$ and $\alpha < 1/2 ~ \Longrightarrow \prod_{n = 1}^{\infty} [1 - 3(2\alpha)^{L+n}  ] > 0$, the proof is concluded.  
\end{proof}

\section{Proofs for the non-decreasing case}\label{sec:nondecreasing}
For a shortened notation, let $\bar0 = (0,0,0, \dots)$ and $\bar1 = (1,1,1, \dots)$ denote the words consisting entirely of $0$'s and $1$'s, respectively.
\begin{lemma}\label{lem:xi_0xi_1}
    Suppose $\mathcal{M}$ is a non-decreasing sequence. Then,
    \[\{\bar 0, \bar 1\}\subset W_0\iff W_0=\Xi . \]
\end{lemma}

\begin{proof}
    The implication $W_0=\Xi \implies \{ \bar 0 , \bar 1\}\subset W_0$ is straightforward. We now deal with the other implication. Suppose that $\{ \bar 0 , \bar 1\}\subset W_0$. We prove by contradiction that the following statement holds:
    \begin{itemize}
        \item (S) for every $v\in\mathbb{Z}_+$, there exists $v_0>v$ and $v_1>v$ such that $X_{v_0}=0$, $X_{v_1}=1$, and $\max\{v_0-v,v_1-v\}\leq M_v$.
    \end{itemize}
    
    In other words, statement S is equivalent to all vertices (including zero) being able to see at least one $0$ and one $1$. By a simple induction argument, one can note that S implies that any word $(\xi_1,\xi_2,...)\in \Xi$ is seen and so $W_0=\Xi$.

    Consider that S does not hold. Then, there exists some $v\in\mathbb{Z}_+$ that does not admit $v_0$ or $v_1$ that satisfies the conditions in S. We deal with the first case (no $v_0$) as the other one is analogous. Since $\mathcal{M}$ is non-decreasing, for any $0\leq v'\leq v$ it also does not exist $v_0$ such that $X_{v_0}=0$ and $v_0-v'\leq M_{v'}$, i.e., neither $v$ nor any previous vertex is able to see a 0. This causes the word $\bar 0$ not to be seen, which is a contradiction.
\end{proof}

\begin{proof}[\textbf{Proof of Theorem \ref{teo:least_probable}}]
The result that $P_{\bar\G}(W_0=\Xi)>0 \Rightarrow P_{\bar\G}(\xi^*\in W_0)>0 $ comes directly from the fact that $\{W_0=\Xi\}\subset \{\xi^*\in W_0\}$, so it only remains to show the other implication. 

Suppose $P_{\bar\G}(\xi^*\in W_0)>0$. By Proposition~\ref{prop:worst}, we have that both $P_{\bar\G}( \bar 0\in W_0)>0$ and $P_{\bar\G}(\bar 1\in W_0)>0$. Our strategy is to use this fact to prove that $P_{\bar\G}( \{\bar 0,\bar 1\}\subset W_0)>0$ and then conclude the proof by using Lemma \ref{lem:xi_0xi_1}.

For $i\in\{0,1\}$ and $v\in\mathbb{Z}_+$, let
\[A_i(v):=\{\text{$(i,i,...)$ is seen from vertex $v$ and $X_v=i$}\}\]
and
\[A_i:=\bigcup_{v\in \mathbb{Z}_+}A_i(v).\]
Note that
\[A_i\in \bigcap_{n\geq 1}\sigma(X_n,X_{n+1},...),\] 
and so $P_{\bar\G}(A_i)\in\{0,1\}$ by Kolmogorov 0-1 law. However, $P_{\bar\G}(A_i)>0$ as $P_{\bar\G}((i,i,...)\in W_0)>0$, so we conclude that $P_{\bar\G}(A_0)=P_{\bar\G}(A_1)=1$. 

Let $V_i$ denote the first vertex from which $(i,i,...)$ is seen and $X_{V_i}=i$. So, $\{V_0<\infty, V_1<\infty\}$ holds almost surely as it is the intersection of the almost sure events $A_0$ and $A_1$. Since $X_{V_0}=0\neq 1=X_{V_1}$ implies $V_0\neq V_1$, at least one of the following holds: $P_{\bar\G}(V_0<V_1<\infty)>0$ or $P_{\bar\G}(V_1<V_0<\infty)>0$. From now on, we continue the proof for the first case and the other one is analogous.

For $v_0<v_1$, we say that a configuration $(X_n)_{n \in \mathbb{Z}_+}$ is ($v_0,v_1$)-good if
\begin{itemize}
    \item $\bar 1$ is seen from $v_1$ and $X_{v_1}=1$;
    \item $X_{v_0}=0$;
    \item there exists $v>v_1$ such that $X_v=0$, $M_{v_0}\geq v-v_0$, and $\bar 0$ is seen from $v$ (and consequently from $v_0$).
\end{itemize}

In other words, in the definition above, we not only require $(i,i,...)$ to be seen from $v_i$ for $i\in\{0,1\}$, but also that $v_0$ does not need help from any other vertex in $\{v_0+1,...,v_1\}$ to see $\bar 0$. Consider the event
\[G(v_0,v_1):=\{(X_n)_{n \in \mathbb{Z}_+}\text{ is $(v_0,v_1)$-good}\}\]
and note that it only depends on $X_n$ for $n\in\{v_0\}\cup\{v_1,v_1+1,...\}$.

When $\{V_0<V_1<\infty\}$ holds, there is an infinite sequence $V_0=x_0,x_1,x_2,...$ of 0's that connect to each other, and so, $(X_n)_{n \in \mathbb{Z}_+}$ is $(v_0,v_1)$-good if we take $v_1=V_1$ and $v_0=x_i$ such that $x_i<V_1<x_{i+1}$. Therefore,
\[0<P_{\bar\G}(V_0<V_1<\infty) \leq P_{\bar\G}\Bigg(\bigcup_{1\leq v_0<v_1<\infty}G(v_0,v_1)\Bigg)\leq \sum_{1\leq v_0<v_1<\infty}P_{\bar\G}\big(G(v_0,v_1)\big)\]
and we conclude that there exists at least one pair $(v_0,v_1)$ such that $v_0<v_1$ and $P_{\bar\G}(G(v_0,v_1))>0$. From now on, consider $(v_0,v_1)$ to be such pair.

Let $E_0:=\{n\in\mathbb{N}:\text{ $n+v_0$ is an even number and } n<v_0\}$ and $E_1:=\{1,2,...,v_1-1\}\setminus(E_0\cup\{v_0\})$ be sets of indices and
\[E:=\{X_n=0\text{ for all $n\in E_0$ and $X_n=1$ for all $n\in E_1$}\}\]
be the event in which $X$ alternates between $0$ and $1$ before $v_0$ and then equals $1$ from $v_0+1$ until $v_1-1$.

Note that $P_{\bar\G}(E)>0$ and that $E$ and $G(v_0,v_1)$ are independent events as they rely on different coordinates of $(X_n)_{n \in \mathbb{Z}_+}$. So, $P_{\bar\G}(E\cap G(v_0,v_1))>0$ and it only remains to note that $E\cap G(v_0,v_1)$ implies that both $\bar 0$ and $\bar 1$ are seen to conclude the proof by Lemma \ref{lem:xi_0xi_1}. Indeed, for any $i\in\{0,1\}$, we have that $(i,i,i,...)$ is seen from $v_i$ and since $M_n\geq 2$ for all $n\in\mathbb{Z}_+$, it is possible to connect vertex $0$ all the way through vertex $v_i$ by using only the vertices in $E_i$ so that $(i,i,i,...)$ is also seen from vertex $0$.
\end{proof}

\begin{proof}[\textbf{Proof of Theorem~\ref{theo:mdeterministic}}] 

We consider only the case $p\geq 1/2$ and so $C=-[\log(p)]^{-1}$ as the case $p<1/2$ is analogous. By relying on Theorem \ref{teo:least_probable}, we only need to consider $\xi^*= \bar 0$ and evaluate $P_{\bar\G}(\xi^*\in W_0)$.

We define spacing random variables as done in \cite{grimmet_liggett_richthammer_2010}. Let $T_0=0$ and define recursively
\[T_{n+1}:=\min\{k> T_n:X_k=0\}, \hspace{.5cm} \tau_{n+1}:=T_{n+1}-T_n.\]

Note that $T_n=\sum_{i=1}^n\tau_{i}$ and that $\tau_i \stackrel{i.i.d.}{\sim}Geo(1-p)$, where $Geo$ denotes the geometric distribution with support on $\mathbb{N}=\{1,2,3,...\}$. We now argue that
\begin{equation}\label{eq:w(m)_subset}
\{\xi^*\in W_0\}= \bigcap_{n\geq 0}\{\tau_{n+1}\leq M_{T_n}\}.
\end{equation}

On the one hand, it is clear that $\{\xi^*\in W_0\}\supset \bigcap_{n\geq 0}\{\tau_{n+1}\leq M_{T_n}\}$. On the other hand, if there exists $n\in\mathbb{N}$ such that $\tau_{n+1}> M_{T_n}$, then the $(n+1)$-th element of $\xi^*$ is not within distance $M_{T_n}$ from $T_n$ (since $\mathcal{M}$ is non-decreasing, the $(i+n)$-th element is also not seen by any of the previous points $T_{n-1},T_{n-2},...,T_{0}$) and so $\xi^*$ is not seen.
\bigskip

\noindent\textit{Proof of Item 1}. By (\ref{eq:w(m)_subset}) and the fact that $\{W_0=\Xi\}\subset \{\xi^*\in W_0\} $, the proof is complete if we show that $P_{\bar\G}(\bigcap_{n\geq 0}\{\tau_{n+1}\leq M_{T_n}\})=0$. The main idea is to use the second Borel-Cantelli lemma to prove that $P_{\bar\G}(\tau_{n+1}>M_{T_n} \text{ infinitely often})=1$. However, some modifications are needed as we need a divergent sum of probabilities of independent events. Since $(T_n)_{n\geq 0}$ is not an independent sequence, we start by temporarily changing $T_n$ to the deterministic value $\lceil \frac{2}{1-p}\rceil n$.

Let $b=b(p):=\lceil \frac{2}{1-p}\rceil$. Since $M_{bn}\le C\log(bn)$ and $C\log(p)=-1$ for all $n \geq n_0$, we have that
\[
\sum_{n=1}^\infty P_{\bar\G}(\tau_{n+1}>M_{bn}) = \sum_{n= 1}^\infty p^{M_{bn}} \ge \sum_{n=n_0}^\infty p^{C\log(bn)}=b^{C\log(p)}\sum_{n= n_0}^\infty n^{C\log(p)}=+\infty.
\]
and so, by the Borel-Cantelli lemma as $(\tau_n)_{n\geq 1}$ is an independent sequence,
\begin{equation}\label{eq:tau_io}
P_{\bar\G}(\tau_{n+1}>M_{bn} \text{ infinitely often})=1.
\end{equation}

Recall that $T_n=\sum_{i=1}^n\tau_{i}$ with $\tau_i \stackrel{i.i.d.}{\sim}Geo(1-p)$ and $E(T_n)=\frac{n}{1-p}\leq \frac{bn}{2}$. So, we use exponentially decaying bounds for the tail of the sum of geometric variables (e.g. the Chernoff bounds found in \cite[Theorem 2.1]{tail_geo}) to conclude that there exists some $a=a(p)$ such that $P_{\bar\G}(T_n\geq bn)\leq P_{\bar\G}(T_n\geq 2E(T_n))\leq e^{-an}$ for any $n\in \mathbb{N}$. This implies that $\sum_{n=1}^\infty P_{\bar\G}(T_n\geq bn)<\infty$ and by the Borel-Cantelli lemma that
\begin{equation}\label{eq:t_io}
P_{\bar\G}\Bigg(\bigcup_{n\geq 0} \bigcap_{k\geq n}\{ T_k <bk\}\Bigg)=P_{\bar\G}\big(\{T_n \geq bn \text{ infinitely often}\}^c\big)=1.
\end{equation}

Now note that $\{\tau_{n+1}>M_{bn}\}\cap\{T_n <bn\}\subset \{\tau_{n+1}>M_{T_n}\}$ as $\mathcal{M}$ is non-decreasing. By (\ref{eq:tau_io}) and (\ref{eq:t_io}), $\{\tau_{n+1}>M_{bn}\}$ occurs infinitely often and $\{ T_n <bn\}$ occurs for all but some finitely many times with probability $1$, so their intersection occurs infinitely many times with probability $1$. Therefore,
\[
P_{\bar\G}(\tau_{n+1}>M_{T_n} \text{ infinitely often})=1,
\]
completing the proof of item 1. Moreover, note that the same exact proof also works when $M_n=C\log(n)$ for all sufficiently large \(n\), what yields the result claimed in Remark~\ref{remark:m_nondecreasing}.
\bigskip

\noindent \textit{Proof of item 2}. In order to be able to use Theorem \ref{teo:least_probable}, we first note that we may assume without loss of generality that $\mathcal{M}$ is non-decreasing. Indeed, for a general $\mathcal{M}$, define $\mathcal{M}'$ by $M_n'=\inf\{M_n,M_{n+1},...\}$. It is clear that $\mathcal{M}'$ is non-decreasing and that $\liminf_{n\to \infty} \frac{M_n}{\log(n)}=\liminf_{n\to \infty} \frac{M'_n}{\log(n)}$. Moreover, since $M'_n\le M_n$, a simple coupling argument shows that it is sufficient to prove that the event $\{W_0=\Xi\}$ has positive probability when considering the graph formed by $\mathcal{M}'$ to conclude the same for the one formed by $\mathcal{M}$.

By (\ref{eq:w(m)_subset}) and Theorem \ref{teo:least_probable}, the proof is complete if we show that $P_{\bar\G}(\bigcap_{n\geq 0}\{\tau_{n+1}\leq M_{T_n}\})>0$. Moreover, since $\mathcal{M}$ is non-decreasing and $T_n\geq n$ by definition, it is enough to prove that $P_{\bar\G}(\bigcap_{n\geq 0}\{\tau_{n+1}\leq M_{n}\})>0$. We have that
\begin{equation}\label{eq:series}
\sum_{n\geq 0}P_{\bar\G}(\tau_{n+1}> M_n)=\sum_{n\geq 0}p^{M_n}.
\end{equation}

The first step of the proof is to show that the series in (\ref{eq:series}) is finite. There exists $n_0\in\mathbb{N}$ such that $M_n= \nu\log(n)$ for $n\geq n_0$. Note that the following inequality follows from $\nu\log(p)<-1$:
\begin{equation}\label{eq:series_p_m}
\begin{aligned}\sum_{n\geq n_0}p^{M_n}&\leq \sum_{n\geq n_0}p^{\nu\log(n)}=\sum_{n\geq n_0}\exp[\log(p)\nu\log(n)]=\sum_{n\geq n_0}n^{\nu\log(p)}<\infty.
\end{aligned}\end{equation}

Thus, if we let $A_n:=\{\tau_{n+1}\leq M_n\}$, then
\[P_{\bar\G}\Bigg(\bigcup_{k= 0}^\infty \bigcap_{n= k}^\infty A_n\Bigg)=1\]
by (\ref{eq:series}), (\ref{eq:series_p_m}) and the Borel-Cantelli Lemma. So, there exists $k\in\mathbb{N}$ such that 
\begin{equation}\label{eq:an_n>k}
P_{\bar\G}\Bigg(\bigcap_{n=k}^\infty A_n\Bigg)>0.
\end{equation}

Since $k$ is finite and $M_n\geq 2$ for all $n\geq 0$, it is also easy to see that
\begin{equation}\label{eq:an_n<k}
P_{\bar\G}\Bigg(\bigcap_{n=0}^{k-1}A_n\Bigg)\geq P_{\bar\G}\Bigg(\bigcap_{n=0}^{k-1}\{\tau_{n+1}\leq 2\}\Bigg)>0
\end{equation}

Combining (\ref{eq:an_n>k}) and (\ref{eq:an_n<k}) with the fact that $A_n$ are independent results in
\[\begin{aligned}
P_{\bar\G}\Bigg(\bigcap_{n\geq 0}A_n\Bigg)&\geq P_{\bar\G}\Bigg(\bigcap_{n\geq 0}A_n\bigg|\bigcap_{n\geq k}A_n\Bigg)P_{\bar\G}\Bigg(\bigcap_{n\geq k}A_n\Bigg)\\
&=P_{\bar\G}\Bigg(\bigcap_{n=0}^{k-1}A_n\Bigg)P_{\bar\G}\Bigg(\bigcap_{n\geq k}A_n\Bigg)\\
&>0,
\end{aligned}\]
completing the proof.
\end{proof}

\section{Exponential growth of the clusters}\label{sec:lemmas}

In this section, we prove Proposition \ref{lemma:exp_word}. We first state an auxiliary proposition depending only on the random graph $\mathcal{G}$ and show that proving this proposition is sufficient to establish Proposition \ref{lemma:exp_word}.

Define the sequence $(i_n)_{n\ge 0}$ by setting $i_0:=0$, $i_1:=M_0$, and recursively for $n\geq 1$,
\[i_{n+1}:=\begin{cases}
    \max\{i+M_i:i\in (i_{n-1}, i_n]\},&\text{if  $i_n\neq i_{n-1}$};\\
    i_n,&\text{otherwise}.
\end{cases}\]

We also define, for $n\ge 1$,
\[C_n:= i_{n} - i_{n-1} .\]
Note that the definition of $C_n$ is similar to that of $C_n(\xi)$, but it does not depend on any specific word or $(X_n)_{n\ge 1}$, so it only depends on the random graph $\mathcal{G}$.

\begin{proposition}~\label{lemma:exp}
Given a parameter $p \in (0,1/2]$, suppose that the common distribution of the ranges $R$ satisfies
\[
\liminf_{n\to\infty} nP(R\geq n) > {3}.
\]
Then there exists a positive constant $\alpha < 1/2$ and parameters $(L,\Phi)$ such that
\begin{equation}\label{eq:prop_aux}
P(C_{b_{n+1}}\ge r_{n+1} \mid C_{b_n}=r ) \geq 1 - 3\alpha^{L+n},
\quad \text{for every $r\ge r_n$ and $n \geq 1$.}
\end{equation}
\end{proposition}
\begin{proof}[\textbf{Proof of Proposition \ref{lemma:exp_word}}] Consider $R$ such that $\liminf_{n\to\infty} nP(R\geq n) > {3}/p$.
For each $n \in \Z_{+}$, let $\bar M_n = \mathbbm{1}_{\{X_n = 1\}}M_n$ and consider $\bar{R}$ to be its representative range distribution.
Since $P(\bar R\ge n)=pP(R\ge n)$, a simple coupling argument shows that (\ref{eq:prop_aux}) 
implies 
\[
\PP(\mathcal{A}_{n+1}(\bar1) \mid C_{b_n}(\bar{1})=r ) \geq 1 - 3\alpha^{L+n},
\quad \text{for every $r\ge r_n$ and $n \geq 1$,}
\]
for some $\alpha<1/2$ and parameters $(L,\Phi)$. Therefore, Proposition~\ref{lemma:exp_word} is proved for the word $\bar1$. We now extend this result to a general word $\xi\in \Xi$.

Note that both $(C_n(\xi))_{n\ge 1}$ and $(C_n(\bar 1))_{n\ge 1}$ are Markov Chains. For any $n\in\mathbb{N}$ and given $C_n(\xi)\le C_n(\bar1)$, we can create a coupling such that $M_{i_n(\xi)-k}= M_{i_n(\bar1)-k}$ and $\mathbbm{1}_{(X_{i_n(\xi)-k}=\xi_n)}\ge\mathbbm{1}_{(X_{i_n(\bar1)-k}=0)}$ for every $k\in\{0,1,...,C_n(\bar1)-1\}$, which implies $C_{n+1}(\xi)\le C_{n+1}(\bar1)$. An inductive coupling argument shows that
\[P(\mathcal{A}_{n+1}^c(\xi)|C_{b_n}(\xi)=r)\le P(\mathcal{A}_{n+1}^c(\bar1)|C_{b_n}(\bar1)=r), \quad\text{for every } n,r\in\mathbb{N}.\]
\end{proof}

Now we prove the auxiliary proposition regarding $\mathcal{G}$. As the proof is long and involves several details, some lemmas are stated throughout this section, with their proofs postponed to Section~\ref{subsec:lemmas}.

\begin{proof}[\textbf{Proof of Proposition \ref{lemma:exp}}]

Consider the following distribution
\begin{equation}\label{eq:dist_r}
P(R\geq n)=\begin{cases}
1,&\text{if }n\leq 0\\
1-e^{-\beta/T},& \text{if } 0<n<T;\\
    1-e^{-\beta/n},&\text{if }n\geq T.
\end{cases}
\end{equation}

We can prove Proposition \ref{lemma:exp} only for distributions satisfying $(\ref{eq:dist_r})$ for $\beta>3$ and $T\in\mathbb{N}$, because a coupling argument extends this result to general distributions. Indeed, considering $(\ref{eq:dist_r})$, we have $\lim_{n\to\infty}nP(R\geq n)=\beta$. For any general $R'$ such that $\liminf_{n\to \infty}n P(R'\geq n)=\beta'>3$, we can select $\beta\in(3,\beta')$ and $T$ large enough so that $P(R'\ge n)\ge P(R\ge n)$ for $n\ge T$. This automatically implies $P(R'\ge n)\ge P(R'\ge T)\ge P(R\ge T)=P(R\ge n)$ also for $n\in \{1,2,...,T\}$ and, therefore, $R'\succeq R$.

Now we fix $\beta>3$ and $T\in\mathbb{N}$ and proceed with the proof considering that \((M_i)_{i\in\mathbb{Z}_+}\) is an i.i.d. sequence whose distribution satisfies $(\ref{eq:dist_r})$. Let $Y_1:=C_1$ and, for $n\in\{2,3,...\}$, define
\[Y_n:=\begin{cases}\frac{C_n}{C_{n-1}}, &\text{ if } Y_{n-1}\neq 0;\\
0, &\text{ if } Y_{n-1}=0.
\end{cases}\]

For any $t \in(0, \infty)$, $n\in\{2,3,...\}$, and $s\in\mathbb{N}$, we have that
\begin{equation}\label{eq:cumul_yn}
\begin{aligned}
P(Y_{n}\leq t|C_{n-1}=s)&=P(C_{n}\leq \lfloor st\rfloor |C_{n-1}=s)\\
&=\prod_{j=0}^{s-1}P(R_j<\lfloor st\rfloor+1+j)\\
&=\exp\bigg\{-\beta\sum_{j=f(s,t,T)}^{s-1}\frac{1}{\lfloor st\rfloor+1+j}-\frac{\beta f(s,t,T)}{T}\bigg\},\end{aligned}\end{equation}
where $f(s,t,T):=\max\{0, T-1-\lfloor st \rfloor\}$.

Denoting
\[r(s, t) := \exp\Bigg\{-\beta \Bigg[\sum_{j=0}^{s-1}\frac{1}{\lfloor st \rfloor+1+j}-\int_0^s\frac{1}{\lfloor st \rfloor+1+x}dx\Bigg]\Bigg\},\]
fixing $t>0$ and taking $s$ sufficiently large so that $f(s,t,T)=0$, our expression for $P(Y_n \le t|C_{n-1}= s)$ can be written as
\[P(Y_n \le t|C_{n-1}= s) = \exp\Bigg\{-\beta \int_0^s\frac{1}{\lfloor st \rfloor+1+x}dx \Bigg\} r(s,t)=\bigg(\frac{\lfloor st \rfloor+1}{\lfloor st \rfloor+1+s}\bigg)^\beta r(s,t).\]
Taking the limit when $s\to \infty$ in the equation above, it holds for any $t\in(0,\infty)$ that
\begin{equation}\label{eq:conv_w}
\lim_{s\to \infty} P(Y_n \le t|C_{n-1}= s)=\lim_{s\to\infty}\Bigg[\bigg(\frac{\lfloor st \rfloor+1}{\lfloor st \rfloor+1+s}\bigg)^\beta r(s,t)\Bigg]=\bigg(\frac{t}{t+1}\bigg)^\beta.\end{equation}

Thus, as $s$ goes to infinity, the conditional distribution of $Y_n|C_{n-1}=s$ converges weakly to a random variable Y with distribution function given by

\begin{equation}\label{eq:cumul_beta_prime}
P(Y\leq t)=\begin{cases}
0,&\text{ if }t\leq 0;\\
(\frac{t}{t+1})^\beta,&\text{ otherwise.}
\end{cases}
\end{equation}

The random variable $Y$ has a known distribution: Beta Prime with parameters $\beta$ and $1$. Given any $t \ge 0$, observe that $P(\log Y \ge t) = P(Y \ge e^{t}) = 1- [e^t/(1 + e^t)]^\beta$ and $P(\log Y \le-t) = P(Y \le e^{-t}) = [e^{-t}/(1 + e^{-t})]^\beta = 1/(1 + e^t)^\beta$. Thus, $P(\log Y \ge t) = P(\log Y \le -t)$ if
$\beta = 1$, which implies that $E[\log Y ] = 0$, as $\log Y$ is integrable. Since $P(\log Y \ge t)$ is an increasing function in $\beta$ whilst $P(\log Y \le -t)$ is a decreasing one, we obtain
\[E[\log Y]\begin{cases}
    <0, \quad\text{ if }\beta<1,\\
    >0,  \quad\text{ if }\beta>1.
\end{cases}\]

The general idea of the proof is to use the identity $C_n=\prod_{i=1}^n Y_i$ together with the weak convergence in \eqref{eq:conv_w} to dominate $C_n$ by a multiplication of i.i.d. random variables.

For any $K\in\mathbb{N}$, define the set $A=A(K):=\{1/2^j:j\in\mathbb{N}\}\cup\{1,2,...,K^2\}$, the sequence $(a_j)_{j\in A}$ by $a_j=a_j(K)=j/K$,
and the sequence of disjoint intervals $(I_j)_{j\in A}$ by
\[I_j=I_j(K):=\begin{cases}
(a_{j},a_{2j}],&\text{ if } j\in \{1/2^j:j\in\mathbb{N}\};\\
(a_{j},a_{j+1}],&\text{ if } j\in \{1,...,K^2-1\};\\
(a_{K^2},+\infty),&\text{ if } j=K^2.
\end{cases}\]

Let the random variable $\tilde{Y}$ be defined as
\[\tilde{Y}=\tilde{Y}(K):=\sum_{j\in A} a_j\mathbbm{1}_{(Y\in I_j)}.\]

In other words, $\tilde{Y}$ is a discrete approximation of $Y$ such that $\tilde{Y}\leq Y$ and $P(\tilde{Y}\neq 0)=1$ for every $K\in \mathbb{N}$. Moreover, $\tilde{Y}(K)$ converges almost surely to $Y$ as $K\to \infty$.

Let $(\tilde{Y}_i)_{i\ge 1}$ be a sequence of i.i.d. copies of $\tilde{Y}$ and define
\[\tilde{\mu}=\tilde{\mu}_K:=E(\log(\tilde{Y})).\] The next Lemma helps us to state a value for $K$ that gives a good approximation in order to recover some Chernoff bounds results from $Y$.

\begin{lemma}\label{lemma:y>exp}
For any $\beta>3$, there exist $\tau=\tau(\beta)\in(0,1/2)$, $K\in\mathbb{N}$ and $\theta\in(0,\tilde{\mu})$ such that for all $n\in\mathbb{N}$,
\begin{equation}\label{eq:lemma_y>exp_1}
P\bigg(\prod_{i=1}^n \tilde{Y}_i\leq e^{(\tilde{\mu}-\theta) n}\bigg)\leq \tau^n\end{equation}
and
\begin{equation}\label{eq:lemma_y>exp_2}
P\bigg(\prod_{i=1}^j \tilde{Y}_i\leq e^{-\theta n} \text{ for some }j\in\{1,2,...,n\}\bigg)\leq \tau^n.\end{equation}
\end{lemma}

Recall that we are dealing with the sequence $b_n$ depending on $L$ which is yet to be defined. We can consider that $(b_n)_{n\ge 0}$ defines a sequence of boxes, where the $n$-th box starts at $b_{n-1}+1$ and ends at $b_n$.

Since $C_n=\prod_{i=1}^n Y_i$, inequality (\ref{eq:lemma_y>exp_1}) is directly related to the exponential growth of $(C_j)_{j\ge1}$. Since $Y_{j+1}|C_j=s$ converges weakly to $Y$ (which, in turn, will be approximated by $\tilde{Y})$ as $s$ goes to infinity, (\ref{eq:lemma_y>exp_2}) helps to establish a baseline $s$ for which $C_j$ must always remain above to ensure a good approximation between $Y_{j+1}$ and $\tilde{Y}$.

Note, however, that maintaining the same baseline $s$ for the whole process is not enough. For any fixed $s$, then $P(Y_n=0|C_{n-1}=s)>0$, which would cause the chain to eventually enter in the absorbing state $0$ and ruin the approximation. Therefore, this baseline cannot be uniform and must depend on $n$.

Figure \ref{fig:c_n} is a pictoral representation of the process. Since the vertical axis is in logarithmic scale, we can note that both $(C_n)_{n\ge1}$ and $(s_n)_{n\ge1}$ grow at least exponentially. The sizes of the boxes, however, grow linearly as $b_{n+1}-b_n=L+n$. Moreover, note that, in each box, the collective ranges are always above the respective baseline, which is crucial for maintaining good control over the approximation between $Y_j$ and $\tilde{Y}$, and for ensuring that the approximation improves (as it holds for more intervals $I_j(K)$) as $n$ increases.

\begin{figure}[h]
\centering
\includegraphics[width=15cm]{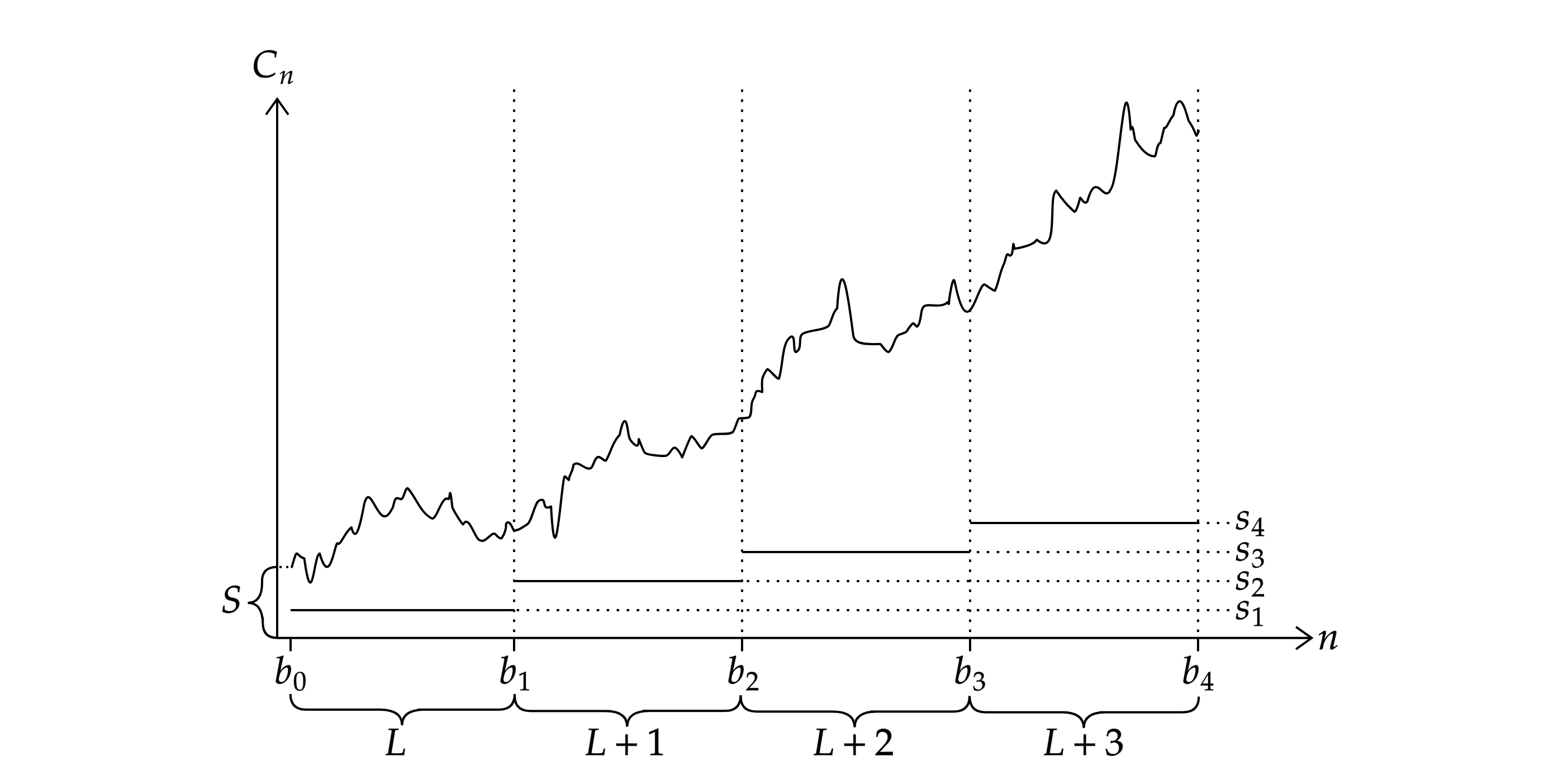}
\caption{Pictoral representation of the growth of the collective ranges. Vertical axis is in logarithmic scale.}\label{fig:c_n}
\end{figure}

We define the sequence of baselines $(s_n)_{n\ge1}$ by setting
\begin{equation}
\label{eq:def_sn}s_1:=Se^{-\theta L} \qquad\text{and} \qquad s_n:=Se^{-\theta (L+n-1)}\prod_{j=0}^{n-2}e^{(\tilde{\mu}-\theta)(L+j)} \text{ for }n\geq 2. \end{equation}

Equivalently, $(s_n)_{n\ge1}$ is the sequence that satisfies the recursive formula $s_{n+1}=s_ne^{(\tilde{\mu}-\theta)(L+n-1)-\theta}$ with $s_1=Se^{-\theta L}$. By adjusting $S$ and $L$, we can control both the inital value $s_1$ as well as the rate at which the baselines increase. The next Lemma characterizes the set of intervals in $\cup_{j\in A(K)}I_j$ for which $Y_j$ is well approximated by $\tilde{Y}$, depending on these baselines.

\begin{lemma}\label{lemma:approx}
    For any $\beta>3$, $\tau\in(0,1/2)$, $K\in\mathbb{N}$, $\theta\in(0,\tilde{\mu})$ and $\delta\in(1,1/\tau)$, there exist $L\in\mathbb{N}$ and $S\in\mathbb{N}$ such that the following items hold:
    \begin{itemize}
        \item[(I)] There exists a sequence $(d_n)_{n\in\mathbb{N}}$ such that for any $n\in\mathbb{N}$:
        \begin{equation}\label{eq:lema_dn}
            d_n\leq \tau^{L+n},\end{equation}
        and for any $s\in\mathbb{N}$ such that $ s\geq s_n$,
        \begin{equation}\label{eq:lema_dn2}
        P\Big(Y_2\leq \frac{1}{K2^{L+n}}\Big|C_1=s\Big)\leq d_{n}\end{equation}
        \item[(II)] For any $s\in\mathbb{N}$ and $j\in A(K)$ such that $s\geq s_n$ and $j\geq (\frac{1}{2})^{L+n}$
    \begin{equation}\label{eq:lema(ii)}
   P(Y_2\in I_{j}|C_1= s)\leq \delta P(\tilde{Y}=a_{j}).\end{equation}
    \end{itemize}
\end{lemma}

Note that, as $n$ increases, the approximation shown in (\ref{eq:lema(ii)}) holds for more $I_j$ accordingly. Moreover, since $(0,\frac{1}{K2^{L+n}}]$ corresponds to the union of $I_j$ of every $j\in A$ such that $j<(\frac{1}{2})^{L+n}$, (\ref{eq:lema_dn}) and (\ref{eq:lema_dn2}) provides control over the probability that $Y_j$ falls within an interval where the approximation does not hold.

Consider any $\beta>3$ and fix $\tau\in(0,1/2)$, $K\in\mathbb{N}$ and $\theta\in(0,\tilde{\mu})$ satisfying Lemma \ref{lemma:y>exp}. Select $\delta>1$ sufficiently close to $1$ so that 
\begin{equation}\label{eq:delta_tau}
\alpha=\delta \tau<1/2.\end{equation}
and fix $L,S\in\mathbb{N}$ satisfying Lemma \ref{lemma:approx}.
Define \[r_n:=Se^{(\tilde{\mu}-\theta)(b_n-1)}=\begin{cases}
    S, \text{ for $n=0$}\\
    S\prod_{j=0}^{n-1}e^{(\tilde{\mu}-\theta)(L+j)} \text{, for $n\in\mathbb{N}$.}
\end{cases}\]

Given that $C_{b_n}=r\ge r_n$, we want to estimate the event in which $C_{b_{n+1}}\ge r_{n+1}$. We consider three distinct ways in which a “bad” event can occur:

\begin{itemize}
    \item simply, $C_{b_{n+1}}<r_{n+1}$ occurs;
    \item $C_j<s_n$ occurs for some $j\in{b_n+1,\ldots,b_{n+1}-1}$, in which case our approximation fails because some collective ranges fall below the baseline;
    \item $Y_j\in\left(0,\frac{1}{K2^{L+n}}\right]$ occurs for some $j\in{b_n+1,\ldots,b_{n+1}-1}$, in which case we fall into an interval where the approximation fails.
\end{itemize}

For any $n\in\mathbb{N}$, define
\[ \Lambda_{n}:=\left\{\begin{array}{ll}(\lambda_1,...,\lambda_{L+n})\in A(K)^{L+n}:& \prod_{i=1}^{L+n} a_{\lambda_j}\leq e^{(\tilde{\mu}-\theta)(L+n)}, \\
        &\vspace{-0.2cm}\\
        & \lambda_j\geq \frac{1}{2^{L+n}} \text{ for all } j\in\{1,...,L+n\}\\
        &\vspace{-0.2cm}\\
        &\text{and }\prod_{i=1}^j a_{\lambda_i}\geq e^{-\theta(L+n)}\text{ for all } j\in\{1,...,L+n\}
        \end{array}\right\}.\]
        
Note that if $C_{b_n}\geq r_n$, then the two following statements hold:
\begin{itemize}
    \item $C_{b_{n+1}}<r_{n+1}$ implies that $\prod_{i=b_n+1}^{b_{n+1}}Y_i<e^{(\tilde{\mu}-\theta)(b_{n+1}-b_n)}=e^{(\tilde{\mu}-\theta)(L+n)}$;
    \item $\prod_{i=b_n+1}^j Y_{i}\geq e^{-\theta(L+n)}$ implies that $C_j\geq r_ne^{-\theta(L+n)}=s_n$.
\end{itemize}

Therefore, $\Lambda_n$ represents the set in which the exponential growth of $(C_n)_{n\ge1}$ possibly fails on the $n$-th box, but the approximation given in (\ref{eq:lema(ii)}) holds for $Y_{b_n+1},...,Y_{b_{n+1}}$.

For any $k\in\{1,...,L+n\}$, define
\[
\Gamma_k:=\left\{\begin{array}{ll}(\gamma_1,...,\gamma_{k})\in A(K)^{k}:&\prod_{i=1}^ka_{\gamma_i}< e^{-\theta(L+n)}\text{, but } \\
&\vspace{-0.2cm}\\
&\gamma_j\geq \frac{1}{2^{L+n}} \text{ for all } j\in\{1,...,k\}\text{ and }\\
&\vspace{-0.2cm}\\
&\prod_{i=1}^j a_{\gamma_i}\geq e^{-\theta(L+n)}\text{ for all } j\in\{1,...,k-1\}\end{array}\right\}\]
and
\[\Omega_k:=\left\{\begin{array}{ll}(\omega_1,...,\omega_{k})\in A(K)^{k}:&\omega_k<\frac{1}{2^{L+n}} \text{, but }\\
&\vspace{-0.2cm}\\
&\omega_j\geq \frac{1}{2^{L+n}} \text{ and }\prod_{i=1}^j a_{\omega_i}\geq e^{-\theta(L+n)}\text{ for all } j\in\{1,...,k-1\}\end{array}\right\}.\]

These sets represent some problematic cases where, respectively, $C_{b_n+k}$ is possibly smaller than the baseline $s_n$; or $Y_{b_n+k}\in I_j$ with $j<(\frac{1}{2})^{L+n}$.

By considering every possible bad scenario in $\Lambda_n\cup (\cup_{k=1}^{L+n}\Gamma_k)\cup (\cup_{k=1}^{L+n}\Omega_k)$, we calculate, for $r\ge r_n$,
\begin{equation}\label{eq:3parts}
\begin{aligned}
P(C_{b_{n+1}}<r_{n+1}|C_{b_{n}}= r)&\leq \sum_{(\lambda_1,...,\lambda_{L+n})\in\Lambda_{n}} P\bigg(\bigcap_{k=b_n+1}^{b_n+L+n}\{Y_k\in I_{\lambda_k}\}\bigg|C_{b_n}=r\bigg)\\
&+\sum_{k=1}^{L+n}\sum_{(\gamma_1,...,\gamma_k)\in\Gamma_k} P\bigg(\bigcap_{j=b_n+1}^{b_n+k}\{Y_j\in I_{\gamma_j}\}\bigg|C_{b_n}= r\bigg)\\
&+\sum_{k=1}^{L+n}\sum_{(\omega_1,...,\omega_k)\in\Omega_k} P\bigg(\bigcap_{j=b_n+1}^{b_n+k}\{Y_j\in I_{\omega_j}\}\bigg|C_{b_n}=r\bigg)
\end{aligned}\end{equation}

We now deal with the three terms above individually, starting from the upper one. Using the chain rule property (i.e. $P(A_1\cap...\cap A_k)=P(A_1)\prod_{j=2}^k P(A_j|\cap_{i=1}^{j-1}A_i)$ for any sequence of events $A_1,...,A_k$) and the fact that $C_j\geq s_n$ for every $j\in\{b_n+1,...,b_n+L+n\}$ when summing over $\Lambda_n$ and considering the conditional $C_{b_n}=r\ge r_n$, we have by (\ref{eq:lema(ii)}) and (\ref{eq:lemma_y>exp_1}) that

\[
\begin{aligned}
\sum_{(\lambda_1,...,\lambda_{L+n})\in\Lambda_{n}} P\bigg(\bigcap_{k=b_n+1}^{b_n+L+n}\{Y_k\in I_{\lambda_k}\}\bigg|C_{b_n}= r\bigg)&\leq \delta^{L+n}\sum_{(\lambda_1,...,\lambda_{L+n})\in\Lambda_{n}} \prod_{j=1}^{L+n} P(\tilde{Y}_j=a_{\lambda_j})\\
&\leq \delta^{L+n}P\bigg(\prod_{j=1}^{L+n}\tilde{Y}_j\leq e^{(\tilde{\mu}-\theta)(L+n)}\bigg)\\
&\leq (\delta\tau)^{L+n}.
\end{aligned}\]

We use (\ref{eq:lemma_y>exp_2}) and the same reasoning on the second term to obtain
\[
\begin{aligned}
\sum_{k=1}^{L+n}\sum_{(\gamma_1,...,\gamma_k)\in\Gamma_k} P\bigg(\bigcap_{j=b_n+1}^{b_n+k}\{Y_j\in I_{\gamma_j}\} &\bigg|C_{b_n}= r\bigg) \\ &\leq \sum_{k=1}^{L+n} \delta^k \sum_{(\gamma_1,...,\gamma_k)\in\Gamma_k} \prod_{j=1}^{k} P(\tilde{Y}_j=a_{\gamma_j})\\
&\leq \delta^{L+n}\sum_{k=1}^{L+n} \sum_{(\gamma_1,...,\gamma_k)\in\Gamma_k} \prod_{j=1}^{k} P(\tilde{Y}_j=a_{\gamma_j})\\
&\leq \delta^{L+n} P\bigg(\prod_{i=1}^j \tilde{Y}_i\leq e^{-\theta (L+n)} \text{ for some }j\in\{1,...,L+n\}\bigg)\\
&\leq (\delta\tau)^{L+n}.
\end{aligned}\]

For the third term, we use similar reasoning. However, unlike the second case, there is a problem with the approximation exactly for the $k$-th element, which is solved by using (\ref{eq:lema_dn2}). We also use (\ref{eq:lema_dn}) to obtain

\[
\begin{aligned}
\sum_{k=1}^{L+n}\sum_{(\omega_1,...,\omega_k)\in\Omega_k} P\bigg(\bigcap_{j=1}^{k}\{Y_j\in I_{\omega_j}\}\Big| C_{b_n}= r\bigg) &\leq d_n\sum_{k=1}^{L+n} \sum_{(\omega_1,...,\omega_{k})\in\Omega_k} P\bigg(\bigcap_{j=1}^{k-1}\{Y_j\in I_{\omega_j}\}\Big|C_{b_n}= r\bigg)\\
&\leq d_n\\
&\leq \tau^{L+n}.
\end{aligned}\] 

By controlling each one of the three terms in the sum in (\ref{eq:3parts}), we have for each $r\ge r_n$,
\[P(C_{b_{n+1}}<r_{n+1}|C_{b_{n}}= r)\le 3(\delta\tau)^{L+n}\le 3\alpha^{L+n},\]
and the proof is complete.
\end{proof}

\subsection{Proof of auxiliary lemmas}\label{subsec:lemmas}

\begin{proof}[\textbf{Proof of Lemma \ref{lemma:y>exp}}]

Since Lemma \ref{lemma:y>exp} regards just the existence of a $K\in\mathbb{N}$, we can focus only on cases where $K=2^j$ for some $j\in\mathbb{N}$. This is particularly useful because $\tilde{Y}(1)\leq \tilde{Y}(2^j)$ for every $j\in\mathbb{N}$. 

Note that $|\log(\tilde{Y}(1))|\leq \frac{1}{\tilde{Y}(1)}$ as $\tilde{Y}(1)\leq 1$. We also have for any $j\in\mathbb{N}$ that $[\log(\tilde{Y}(2^j))]^+\leq [\log(Y)]^+$
 and $[\log(\tilde{Y}(2^j))]^-\leq [\log(\tilde{Y}(1))]^-$, where $g^+(x)=\max\{g(x),0\}$ and $g^-(x)=\max\{-g(x),0\}$ represent the positive and negative parts, respectively, of a function $g$.
Thus, we conclude for any $j\in\mathbb{N}$ that 
\begin{equation}\label{eq:|log|}
|\log(\tilde{Y}(2^j))|\leq [\log(\tilde{Y}(1))]^-+ [\log(Y)]^+, \text{ with }E\{[\log(\tilde{Y}(1))]^-+ [\log(Y)]^+\}<\infty.
\end{equation}
 
For any $K$ of the form $K=2^j$ for some $j\in\mathbb{N}$, let $\tilde{\mu}=\tilde{\mu}_K:=E(\log(\tilde{Y}))\notin\{-\infty,+\infty\}$. Define $S_0:=0$ and, for $n\in\mathbb{N}$, $\mathcal{F}_n:=\sigma(\tilde{Y}_{1},...,\tilde{Y}_n)$ and
\[S_n:=\sum_{i=1}^n\big[\tilde{\mu}-\log(\tilde{Y}_i)\big].\] 

Note that $(S_n)_{n\ge0}$ is a martingale with respect to the filtration $\{\mathcal{F}_n\}_{n\ge1}$. Moreover, a convex and integrable function over a martingale determines a submartingale with respect to the same filtration \cite[Theorem 4.2.6]{pte}. We shall use the fact that $(e^{S_n})_{n\ge0}$ is a submartingale with respect to $\{\mathcal{F}_n\}_{n\ge1}$ and, by the convexity of the exponential function, it only remains to prove that it is integrable. Indeed,
\[E(e^{S_n})=\Big[e^{\tilde{\mu}}E\big(\tilde{Y}^{-1}\big)\Big]^n\]
with $0\le \tilde{Y}^{-1}(2^j)\le\tilde{Y}^{-1}(1)$ and, since $1-\beta<0$, we have
    \begin{equation}\label{eq:e_1/y}
\begin{aligned}
E\left(\tilde{Y}(1)^{-1}\right)&=P(Y>1)+\sum_{j=1}^\infty 2^{j} P\bigg(Y\in\Big(\frac{1}{2^{j}},\frac{1}{2^{j-1}}\Big]\bigg)\\
&\leq 1+\sum_{j=1}^\infty 2^{j}\Big(\frac{1}{2^{j-1}}\Big)^\beta\\
&\leq 1+2^\beta \sum_{j=1}^\infty 2^{(1-\beta)j}<\infty.
\end{aligned}\end{equation}

By Doob's Inequality for submartingales (see \cite[Theorem 4.4.2]{pte}), we have that 
\begin{equation}\label{eq:doob}
P\Big(\max_{1\leq j \leq n}S_j\geq \theta n\Big)=P\Big(\max_{1\leq j \leq n}e^{S_j}\geq e^{\theta n}\Big)\leq \frac{E(e^{S_n})}{e^{\theta n}}=\Big[e^{(\tilde{\mu}-\theta)}E\big(\tilde{Y}^{-1}\big)\Big]^n.\end{equation}

Recall that $\tilde{Y}(K)$ converges almost surely to $Y$ as $K\to \infty$ and that $\tilde{Y}(1)\leq \tilde{Y}(2^j)$ for every $j\in\mathbb{N}$. Thus, by (\ref{eq:e_1/y}), the Dominated Convergence Theorem, and Lemma \ref{lem:beta_prime} to conclude that
\begin{equation}\label{eq:lim_mt}
\lim_{j\to\infty}E\Big(\big(\tilde{Y}(2^j)\big)^{-1}\Big)=E\big(Y^{-1}\big)<\frac12.    
\end{equation}

Define \[\tau:=
     \frac{E(Y^{-1})+1/2}{2}\in (E(Y^{-1}),1/2),\]
    as the midpoint between $E(Y^{-1})$ and $1/2$. Moreover, we rely on (\ref{eq:|log|}) to apply the Dominated Convergence Theorem again and show that
\begin{equation}\label{eq:lim_mu}
\lim_{j\to\infty}\tilde{\mu}_{2^j}=\mu\in(0,\infty). 
\end{equation}
Therefore, by (\ref{eq:lim_mt}) and (\ref{eq:lim_mu}) there exists $K\in\mathbb{N}$ satisfying, at the same time, both of the following statements:
\[E\Big(\big(\tilde{Y}(K)\big)^{-1}\Big)
        <\tau\]
        and
        \[\tilde{\mu}_K>0.\]

Since $\tilde{\mu}_K>0$ is already fixed, we select $\theta\in(0,\tilde{\mu})$ by making $e^{(\tilde{\mu}-\theta)}$ sufficiently close to $1$ so that, by (\ref{eq:doob}),
\begin{equation}\label{eq:max_sj}
P\Big(\max_{1\leq j \leq n}S_j\geq \theta n\Big)<\tau^n .   \end{equation}

It just remains to argue that both statements of the lemma come from (\ref{eq:max_sj}). Note that
\begin{equation}\label{eq:set1}
\begin{aligned}
\bigcup_{j \in\{1,...,n\}} \bigg\{\prod_{i=1}^j \tilde{Y}_i\leq e^{-\theta n}\bigg\} &\subset  \bigcup_{j \in\{1,...,n\}} \bigg\{\prod_{i=1}^j \frac{\tilde{Y}_i}{e^{\tilde{\mu}}}\leq e^{-\theta n}\bigg\}\\
&=  \bigcup_{j \in\{1,...,n\}} \{S_j\geq \theta n\}\\
&= \Big\{\max_{1\leq j \leq n} S_j\geq \theta n\Big\}.
\end{aligned}
\end{equation}
Analogously,
\begin{equation}\label{eq:set2}
\begin{aligned}
\bigg\{\prod_{i=1}^n \tilde{Y}_i\leq e^{(\tilde{\mu}-\theta) n}\bigg\} &\subset  \bigcup_{j \in\{1,...,n\}} \bigg\{\prod_{i=1}^j \frac{\tilde{Y}_i}{e^{\tilde{\mu}}}\leq e^{-\theta n}\bigg\}= \Big\{\max_{1\leq j \leq n} S_j\geq \theta n\Big\}.
\end{aligned}
\end{equation}

The proof of the first statement comes from the combination of (\ref{eq:max_sj}) and (\ref{eq:set1}), while the second one comes from the combination of (\ref{eq:max_sj}) and (\ref{eq:set2}).
\end{proof}

\begin{proof}[\textbf{Proof of Lemma \ref{lemma:approx}}]

Before proving items (I) and (II), we have to make some previous calculations and to define $L$ and $S$. Recall that $Y_{n+1}|C_n=s$ converges weakly to $Y$ as $s$ goes to infinity. As $\delta>1$, every $h\in\mathbb{N}$ admits $s^*(h)\in\mathbb{N}$ such that
\begin{equation}\label{eq:cond_y}
P(Y_2\in I_j|C_1=s)\leq\delta P(\tilde{Y}=a_j), \text{ if $s\geq s^*(h)$ and $j\in\{1/2^x:x\in\{1,2,...,h\}\}\cup\{1,2,...,K^2$}\},
\end{equation}
since this only requires asymptotic approximations of the cumulative function $P(Y\leq a_j)$ for finitely many $j$, i.e., those in $\{1/2^x:x\in\{1,2,...,h\}\}\cup\{1,2,...,K^2\}$. This argument however, does not apply to show that (\ref{eq:cond_y}) holds for any $j\in A(K)$, since $A(K)$ has infinitely many elements. The idea is to use some bounds on (\ref{eq:cumul_yn}) to control the convergence speed and the values for \(j\) that make (\ref{eq:cond_y}) to hold as $s$ increases.

Let $H_n:=\sum_{k=1}^n \frac{1}{k}$
be the $n$-th harmonic number. There exists a $\gamma$ (also known as the Euler-Mascheroni constant) such that the following inequalities hold for every $n\in\mathbb{N}$:
    \[\frac{1}{2(n+1)}<H_n-\log(n)-\gamma<\frac{1}{2n}\]
    (see, for instance, \cite{harmonic}).
Using the previous inequalities, one can also found upper and lower bounds for the difference between two harmonic numbers. For every $s,x\in\mathbb{N}$ with $x\geq 2$,
\begin{equation}\label{eq:bounds_hn}
\log\Big(\frac{x+s-1}{x-1}\Big)-\frac{1}{2x-2}\leq\sum_{j=0}^{s-1}\frac{1}{x+j}\leq \log\Big(\frac{x+s-1}{x-1}\Big).
\end{equation}

We now apply these bounds on (\ref{eq:cumul_yn}). When $t\geq 1/s$ and $f(s,t,T)=0$, we have by the upper bound in (\ref{eq:bounds_hn}) that
\begin{equation}\label{eq:lower_y}
\begin{aligned}
    P(Y_{n+1}\leq t|C_n=s)&= \exp\bigg\{-\beta\sum_{j=0}^{s-1}\frac{1}{\lfloor ts\rfloor +1+j}\bigg\}\\
    &\geq \exp\bigg\{-\beta \log\bigg(\frac{\lfloor ts\rfloor+s}{\lfloor ts\rfloor}\bigg)\bigg\}\\
    &= \bigg(\frac{\lfloor ts\rfloor}{\lfloor ts\rfloor+s}\bigg)^{\beta}.
\end{aligned}\end{equation}

On the other hand, if $t\geq 1/s$ and $f(s,t,T)=0$, the following expression comes with the lower bound in (\ref{eq:bounds_hn}):
\begin{equation}\label{eq:upper_y}\begin{aligned}
    P(Y_{n+1}\leq t|C_n=s)&= \exp\bigg\{-\beta\sum_{j=0}^{s-1}\frac{1}{\lfloor ts\rfloor +1+j}\bigg\}\\
    &\leq \exp\bigg\{-\beta \bigg[\log\bigg(\frac{\lfloor ts\rfloor+s}{\lfloor ts\rfloor}\bigg)-\frac{1}{2\lfloor ts\rfloor}\bigg]\bigg\}\\
    &\leq  \Big(\frac{t}{t+1}\Big)^{\beta}\exp\bigg\{{\frac{\beta}{2\lfloor ts\rfloor}\bigg\}}.
\end{aligned}\end{equation}

Since $\delta>1$, we have that $\delta+\frac{1-\delta}{2^\beta}>1$. Select $\zeta>0$ such that

\[c:=\delta-\delta\Big(\frac{1}{2}+\zeta\Big)^\beta+\Big(\frac{1}{2}-\zeta\Big)^\beta>1.\]

Let $(d^*_n)_{n\ge1}$ be a sequence defined by 
\[d^*_n:=\frac{c}{K^\beta2^{\beta n}}.\]

Since $2^\beta\tau>1$, we have that $d^*_n/\tau^n\to 0$ as $n$ approaches infinity. Therefore, we can select $L$ large enough so that each of the following statements hold:
\begin{equation}\label{eq:>-1}
d^*_{L+n}\leq\tau^{L+n}, \text{ for every $n\in\mathbb{N}$},
\end{equation}
\begin{equation}\label{eq:1/2+zeta}
\frac{K+1/2^{L-1}}{2K+1/2^{L-1}}\le \frac{1}{2}+\zeta,
\end{equation}
and
\begin{equation}\label{eq:L>2}
e^{(\tilde{\mu}-\theta)L-\theta}>4.
\end{equation}

We now take $S\in\mathbb{N}$ to be large enough such that each of the following statements holds: 
\[S>e^{\theta L}s^*(L) \text{ with $s^*(L)$ defined in (\ref{eq:cond_y})},\]
\begin{equation}\label{eq:<c_Beta}
\exp\bigg\{\frac{\beta}{2\lfloor Se^{-\theta L}/K2^{L+1}\rfloor}\bigg\}\leq c,\end{equation}
\begin{equation}\label{eq:tau_l_s}
\frac{K2^{L+2}+K^2 4^{L+1}}{2S e^{-\theta L}}<\zeta,
\end{equation}
and
\begin{equation}\label{eq:f(s,t,T)}
\lfloor S e^{-\theta L}/K2^{L+1} \rfloor+1\geq T.\end{equation}

(I)
Define the sequence $(d_n)_{n\ge1}$ by setting
\[d_n:=d^*_{L+n}.\]
The proof of (\ref{eq:lema_dn}) comes directly by the definition of $d_n$ and (\ref{eq:>-1}). 

Our next step is to argue that
\begin{equation}\label{eq:<2_n}
\exp\bigg\{\frac{\beta}{2\lfloor s_n/K2^{L+n}\rfloor}\bigg\}\leq c, \text{ for all } n\in\mathbb{N}.\end{equation}
In order to prove this, we just recall that $s_1=S e^{-\theta L}$ and so (\ref{eq:<2_n}) holds for $n=1$ directly by (\ref{eq:<c_Beta}). Moreover, by (\ref{eq:L>2}), we have that $s_{n+1}=s_ne^{(\tilde{\mu}-\theta) (L+n-1)-\theta}>4s_{n}$. So, a simple induction argument proves (\ref{eq:<2_n}) for every $n\in\mathbb{N}$.

By (\ref{eq:tau_l_s}), (\ref{eq:f(s,t,T)}), and similar arguments, we also have that
\begin{equation}\label{eq:tau_l_s_n}
\frac{K2^{L+1+n}+K^2 4^{L+n}}{2s_n}<\zeta, \text{ for all } n\in\mathbb{N};
\end{equation}
\begin{equation}\label{eq:f(s,t,T)n}
f\Big(s_,\frac{1}{K2^{L+n}},T\Big)=\max\{0, T-1-\lfloor s/K2^{L+n} \rfloor\}=0, \text{ for all  $n\in\mathbb{N}$ and $s\ge s_n$}.
\end{equation}

By (\ref{eq:upper_y}), (\ref{eq:<2_n}) and (\ref{eq:f(s,t,T)n}), we conclude for any $s,n\in \mathbb{N}$ with $s\geq s_n$ that
\[P\Big(Y_2\leq \frac{1}{K2^{L+n}}\Big|C_1=s\Big)\leq\Big(\frac{1}{K2^{L+n}+1}\Big)^\beta \exp\bigg\{\frac{\beta}{2\lfloor s/K2^{L+n}\rfloor}\bigg\} \leq d^*_{L+n}=d_n.\]

(II)
Fix $s,n\in\mathbb{N}$ and $j\in \{1/2^k: k\in\mathbb{N}\}$ such that $s\geq s_n$ and $1/2^{L+n}\leq j\leq 1/2^{L}$. Note that (\ref{eq:1/2+zeta}) implies 
\[\bigg(\frac{2j+K}{2j+2K}\bigg)^\beta\leq \bigg(\frac{1}{2}+\zeta\bigg)^\beta\]
and (\ref{eq:tau_l_s_n}) implies
\[\begin{aligned}
\bigg(\frac{2j+K}{2j}\bigg)^\beta\bigg(\frac{\lfloor js/K \rfloor}{\lfloor js/K \rfloor+s}\bigg)^\beta
&\geq \bigg(\frac{2j+K}{2j}\bigg)^\beta\bigg(\frac{ js/K -1}{js/K+s}\bigg)^\beta\\
&\ge \bigg(\frac{2j+K}{2j+2K}-\frac{2jK+K^2}{2j^2s}\bigg)^\beta\\
&\ge \bigg(\frac{1}{2}-\zeta\bigg)^\beta.
\end{aligned}\]

Therefore, by (\ref{eq:<2_n}),
\[
\exp\bigg\{\frac{\beta}{2\lfloor 2js/K\rfloor}\bigg\}\leq c=\delta-\delta\Big(\frac{1}{2}+\zeta\Big)^\beta+\Big(\frac{1}{2}-\zeta\Big)^\beta\leq \delta-\delta\bigg(\frac{2j+K}{2j+2K}\bigg)^\beta+\bigg(\frac{2j+K}{2j}\bigg)^\beta\bigg(\frac{\lfloor js/K \rfloor}{\lfloor js/K \rfloor+s}\bigg)^\beta.\]
Then, by (\ref{eq:lower_y}), (\ref{eq:upper_y}) and (\ref{eq:f(s,t,T)n}),
\[\begin{aligned}
P(Y_{2}\in I_j|C_1=s)&=P\Big(Y_2\leq \frac{2j}{K}\Big|C_1=s\Big)-P\Big(Y_2\leq \frac{j}{K}\Big|C_1=s\Big)\\
&\leq \Big(\frac{2j}{2j+K}\Big)^{\beta}\bigg[\delta-\delta\Big(\frac{2j+K}{2j+2K}\Big)^\beta+\Big(\frac{2j+K}{2j}\Big)^\beta\Big(\frac{\lfloor js/K \rfloor}{\lfloor js/K \rfloor+s}\Big)^\beta\bigg]-\Big(\frac{\lfloor js/K \rfloor}{\lfloor js/K \rfloor +s}\Big)^{\beta}\\
&=\delta\Big(\frac{2j}{2j+K}\Big)^{\beta} - \delta\Big(\frac{j}{j+K}\Big)^{\beta}\\
&=\delta P(\tilde{Y}=a_j).
\end{aligned}\]

This concludes the proof for the cases where $s,n\in\mathbb{N}$, $j\in A(K)$ with $s\geq s_n$ and ${1}/{2^{L+n}}\leq j\leq {1}/{2^L}$. Note that the remaining cases are also covered by (\ref{eq:cond_y}) and the fact that we defined $S$ in a way such that $s_n\geq s^*(L)$ for all $n\in\mathbb{N}$.
\end{proof}

\begin{appendix}

\section{Appendix: Beta prime}

\begin{lemma}\label{lem:beta_prime}
    For every $\beta>3$,
    \[E(Y^{-1})<1/2.\] 
\end{lemma}

\begin{proof}[Proof of Lemma \ref{lem:beta_prime}]

For $a,b>0$, consider the Beta function
\[B(a,b):=\int_0^1 t^{a-1}(1-t)^{b-1} dt.\]
and the Gamma function
\[\Gamma(a):=\int_0^\infty e^{-t} t^{a-1}dt.\]

It is well known that for $a,b>0$,
\begin{equation}\label{eq:beta_gamma}
    B(a,b)=\frac{\Gamma(a)\,\Gamma(b)}{\Gamma(a+b)},
    \end{equation}
    and that for $a>0$,
    \begin{equation}\label{eq:recur_gamma}
    \Gamma(a+1)=a\,\Gamma(a).
    \end{equation}

Since $Y$ follows a Beta Prime($\beta,1$) distribution with cumulative function given by (\ref{eq:cumul_beta_prime}), its density is given by
\[f_Y(y)=\beta \frac{y^{\beta-1}}{(1+y)^{\beta+1}}\mathbbm{1}_{(y>0)}.\]
Therefore, we compute the integral $E(Y^{-1})$ by substituting $y=\frac{u}{1-u}$ and $dy=(1-u)^{-2}du$, obtaining
\begin{equation}\label{eq:y^t}
E(Y^{-1})=\beta\int_0^\infty \frac{y^{\beta-2}}{(1+y)^{\beta+1}}dy=\beta\int_{0}^1 u^{\beta-2}(1-u)du=\beta B(\beta-1,2)\end{equation}

Since $\Gamma(2)=\Gamma(1)=1$, we have by \eqref{eq:beta_gamma} and \eqref{eq:y^t} that 
\[E(Y^{-1})=\frac{\beta\,\Gamma(\beta-1)\,\Gamma(2)}{\Gamma(\beta+1)}=\frac{1}{\beta-1},\]
where both $\Gamma(\beta+1)=\beta(\beta-1)\Gamma(\beta-1)$ and $\Gamma(2)=\Gamma(1)=1$ follow from \eqref{eq:recur_gamma}.
Therefore, $E(Y^{-1})<1/2$ whenever $\beta>3$.
\end{proof}

\end{appendix}

\section*{Acknowledgements}
G.O.C was partially supported by Conselho Nacional de Desenvolvimento Científico e Tecnológico (CNPq), grant 150251/2026-2. P.A.G. was partially supported by FAPESP, grant 2023/13453-5, grant 2025/27064-6 and by CNPq, grant \textit{Universal} 403423/2023-6. This study was financed in part by the Coordenação de Aperfeiçoamento de Pessoal de Nível Superior - Brasil (CAPES) - Finance Code 001.

%\addbibresource{sample.bib}
%\bibliographystyle{apalike}
%\bibliography{sample}

\printbibliography

\end{document}